\documentclass[12pt]{article}
\usepackage{amssymb,amsmath,amsthm}

\newtheorem{formula}{}
\newtheorem{proposition}[formula]{Proposition}
\newtheorem{corollary}[formula]{Corollary}

\newtheorem{thm}[formula]{Theorem}
\newtheorem{example}[formula]{Example}

\begin{document}
\title{Zeta Functions for tensor products of locally coprime integral adjacency algebras of association schemes}

\author{Allen Herman\thanks{The first authour acknowledges the support of an NSERC Discovery Grant.}, Mitsugu Hirasaka${}^{\dag}$, and Semin Oh\thanks{This research was supported by Basic Research Program through
the National Research Foundation of Korea(NRF) funded by the Ministry of Education,
Science and Technology (grant number NRF-2013R1A1A2012532).}}
\date{February 1, 2016 }
\maketitle

\begin{abstract}
\noindent
The zeta function of an integral lattice $\Lambda$ is the generating function $\zeta_{\Lambda}(s) = \sum\limits_{n=0}^{\infty} a_n n^{-s}$, whose coefficients count the number of left ideals of $\Lambda$ of index $n$.  We derive a formula for the zeta function of $\Lambda_1 \otimes \Lambda_2$, where $\Lambda_1$ and $\Lambda_2$ are $\mathbb{Z}$-orders contained in finite-dimensional semisimple $\mathbb{Q}$-algebras that satisfy a ``locally coprime'' condition.
We apply the formula obtained above to $\mathbb{Z}S \otimes \mathbb{Z}T$ and obtain
the zeta function of the adjacency algebra of the direct product of two finite association schemes $(X,S)$ and $(Y,T)$
in several cases where the $\mathbb{Z}$-orders $\mathbb{Z}S$ and $\mathbb{Z}T$ are locally coprime and
their zeta functions are known.

\end{abstract}

\smallskip
{\small \noindent {\it Key words :} Association schemes, adjacency algebras, zeta functions, integral representation theory.
}

\bigskip
\noindent {\bf 1. Introduction }

\medskip
Suppose $\Lambda$ is a $\mathbb{Z}$-order in a finite-dimensional $\mathbb{Q}$-algebra $A$, $V$ is a finite-dimensional $A$-module, and $L$ is a full $\Lambda$-lattice in $V$; i.e. a free $\mathbb{Z}$-submodule of $V$ with $\mathbb{Q}L = V$.  Solomon's zeta function is the function of a complex variable $s$ defined by
$$ \zeta_{\Lambda}(L,s) = \sum\limits_{N \subseteq L} (L:N)^{-s} $$
where $N$ runs over all full $\Lambda$-sublattices of $L$.  (The reader should be mindful that this was denoted $\zeta_L(s;\Lambda)$ in \cite{Sol77}.)  For each positive integer $n$, let $a_n$ be the number of full $\Lambda$-sublattices of $L$ with index $n$.  Then
$$ \zeta_{\Lambda}(L,s) = \sum\limits_{n=1}^{\infty} a_n n^{-s}. $$
We will not deal here with issues of convergence or continuation of these series, except to say it is well-known that they do converge uniformly on a region in the complex plane where the real part of $s$ is sufficiently large.   In \cite{Sol77}, Louis Solomon establishes several results for $\zeta_{\Lambda}(L,s)$ when $\Lambda$ is an order in a finite-dimensional semisimple $\mathbb{Q}$-algebra, the first of these being an Euler product formula
$$ \zeta_{\Lambda}(L,s) = \prod_p \zeta_{\Lambda_p}(L_p, s). $$
Here $p$ runs through the rational primes $p$, and $\Lambda_p$ and $L_p$ are the $p$-adic completions of $\Lambda$ and $L$, respectively, obtained by tensoring with the ring of $p$-adic integers $\mathbb{Z}_p$.  (Although Solomon uses this notation for $p$-localizations, this does not make a difference to the zeta function - see \cite[Lemma 9]{Sol77}.)  If the semisimple algebra $\mathbb{Q}_p\Lambda_p$ decomposes as $\oplus_{i=1}^h M_{r_i}(D_i)$, where the $p$-adic division algebras $D_i$ occurring in this decomposition have index $m_i$, the center of $D_i$ has ring of integers $R_i$, and each simple component has maximal order $\Gamma_{p,i}$, then we can compute the $p$-adic zeta function $\zeta_{\Gamma_p}(L_p,s)$.  To do so, first decompose the $\mathbb{Q}_p\Lambda$-module $\mathbb{Q}_p \otimes_{\mathbb{Q}} V$ into its homogeneous components.  This expresses it as a direct sum of modules $V_{p,i}$, each of which is isomorphic to a direct sum of $k_i$-copies of an irreducible $M_{r_i}(D_i)$-module, for $i=1, \dots, h$.  The lattice $L_p$ similarly decomposes as the direct sum of $L_{p,i}$'s, each a sublattice of $V_{p,i}$.  According to Hey's formula (see \cite[pg. 145]{BR80})
$$\begin{array}{rcl}
 \zeta_{\Gamma_p}(L_p,s) &=& \prod_{i=1}^h \zeta_{\Gamma_{p,i}}(L_{p,i},s) \\
                                          &=& \prod_{i=1}^h \prod_{j=0}^{k_i-1} \zeta_{R_i}(r_im_is-jm_i).
                                          \end{array}$$
Here $\zeta_{R_i}(s)$ is the usual Dedekind zeta function; i.e.
$$ \zeta_{R}(s) = \prod_{\mathcal{P}} (1 - \mathcal{N}(\mathcal{P})^{-s})^{-1}, $$
where $\mathcal{P}$ runs over all maximal ideals of the Dedekind domain $R$, and $\mathcal{N}(\mathcal{P})$ is the least power of the positive integral prime $p \in \mathcal{P}$ that generates the norm of the ideal $\mathcal{P}$.  For all but finitely many primes $p$, the $\mathbb{Z}_p$-order $\Lambda_p$ is a maximal order of $\mathbb{Q}_p\Lambda$, so the problem reduces to calculating  $\zeta_{\Lambda_p}(L_p, s)$ for these particular primes.

\medskip
In the particular case of the regular left $A$-module $V={}_AA$, we write $\zeta_{\Lambda}(s)$ for $\zeta_{\Lambda}(\Lambda,s)$.  In this case the left $\Lambda$-sublattices of $\Lambda$ of index $n$ are the $\mathbb{Z}$-free left ideals of $\Lambda$ with index $n$.  In \cite{Sol77}, Louis Solomon established several results for the zeta function of the integral group ring $\mathbb{Z}G$ treated as an order in $\mathbb{Q}G$.  Up to now explicit formulas expressing zeta functions of integral group rings in terms of products of Dedekind zeta functions of number fields have been found in a few instances.  The first was the case of the integral group ring $\mathbb{Z}C_p$ of a cyclic group of prime order $p$, whose zeta function is
(see \cite[Theorem~1]{Sol77})
$$ \zeta_{\mathbb{Z}C_p}(s) = (1 - p^{-s} + p^{1-2s})\zeta_{\mathbb{Z}}(s)\zeta_{\mathbb{Z}[\varepsilon_p]}(s), $$
where $\varepsilon_p$ denotes a primitive root of unity of order $p$.

We will be interested in $\mathbb{Z}$-orders in the $\mathbb{Q}$-adjacency algebras of finite association schemes.
An association scheme is a purely combinatorial object generalizing the set of orbitals of a transitive permutation group.
 Let $X$ be a finite set and let $S$ be a partition of $X\times X$ which does not contain the empty set. We say that the 
pair $(X,S)$ is an \textit{association scheme} if
it satisfies the following conditions (see \cite{Ziesh}):
\begin{enumerate}
\item $\{(x,x)\mid x\in X\}$ is a member of $S$;
\item For each $s\in S$, $s^*=\{(x,y)\mid (y,x)\in s\}$ is a member of $S$;
\item For all $s,t,u\in S$, the size of $\{z\in X\mid (x,z)\in s, (z,y)\in t\}$ is constant whenever $(x,y)\in u$.
\end{enumerate}
Each relation $s\in S$ has a corresponding \textit{adjacency matrix} $\sigma_s$ whose rows and columns are indexed by the elements of $X$ as follows:
\[(\sigma_s)_{x,y}=\begin{cases}
1 & \mbox{if $(x,y)\in s$}\\
0 & \mbox{if $(x,y)\notin s$.}
\end{cases}
\]
The pair $(X,S)$ is an association scheme if and only if
\begin{enumerate}
\item The identity matrix is a member of $\{\sigma_s\mid s\in S\}$;
\item $\sum_{s \in S} \sigma_s = J$, the all $1$'s matrix whose rows and columns are indexed by $X$; 
\item for all $s \in S$, $\sigma_{s^*} = \sigma_s^{\top}$, where $\sigma_s^{\top}$ is the transpose of the matrix $\sigma_s$; and 
\item For all $s,t\in S$ the matrix product $\sigma_s\sigma_t$ is a nonnegative integer linear combination of $\{\sigma_s\mid s\in S\}$.
\end{enumerate}
For an association scheme $(X,S)$ we denote the set of $\mathbb{Z}$-linear combinations of $\{\sigma_s\mid s\in S\}$
by $\mathbb{Z}S$, and for a ring $R$ with unity we denote the tensor product $R\otimes_\mathbb{Z} \mathbb{Z}S$ by $RS$.
The above conditions imply that $RS$ is a subalgebra of the full matrix algebra of degree $|X|$ over $R$, which is called the \textit{adjacency algebra} of $(X,S)$ over $R$.   It is well-known that $RS$ is semisimple if $R$ is an extension field of $\mathbb{Q}$ (see \cite{hanakisemi}).   In particular, $\mathbb{Z}S$ is naturally a $\mathbb{Z}$-order in the semisimple algebra $\mathbb{Q}S$.  

Recently formulas for the zeta function for $\mathbb{Z}S$ were produced for some schemes of small rank \cite{HH14}.  We will consider the problem of expressing the zeta function of a tensor product of two integral adjacency algebras.  The approach of the present paper is based on ideas used by Hironaka \cite{Hiron81} to extend the calculation of $\zeta_{\mathbb{Z}C_p}(s)$ to the calculation of $\zeta_{\mathbb{Z}C_n}(s)$ for $n$ a square-free integer.   The methods of this paper will allow us to extend the cases where the zeta function is known to certain direct products of these known cases.

\bigskip
\noindent{\bf 2. Zeta functions for rank $2$ orders over rings of $p$-adic integers}

\medskip
Let $R$ be the ring of integers of a $p$-adic number field $F$. This means the field $F$ is a finite extension of the usual $p$-adic field $\mathbb{Q}_p$, and its ring of integers is a local PID with maximal ideal $\pi R$ for some $\pi \in R$.  The degree of the extension $F/\mathbb{Q}_p$ is $ef$, where $f$ is its residue degree and $e$ is its ramification index.  This means $[R:\pi R] = p^f$ and $\pi^e R = pR$ for the unique integer prime $p$ for which $p \in \pi R$.

Let $\Lambda$ be a $\mathbb{Z}$-lattice of rank $2$.  This means $\Lambda$ is an $R$-free subring of $\mathbb{Q}\Lambda$, and $\Lambda = R1 + Rx$ for some $x \in \Lambda \setminus R1$.  In this case the minimal polynomial of $x$ has degree $2$ in $\mathbb{Z}[x]$.  We will assume the (monic) minimal polynomial of $x$ is {\it reducible} in $\mathbb{Z}[x]$, so it factors as the product of (not necessarily distinct) linear factors.  By a shift of variable we may assume the minimal polynomial of $x$ has the form $x(x-n)$ for some integer $n$.  We wish to calculate the zeta function of $R\Lambda$ under these assumptions, for our $p$-adic ring of integers $R$.

For a nonzero integer $n$ we define $[n]$ to be the maximal positive integer $[n]$ for which $p^{[n]}$ divides $n$.

\begin{thm} \label{zetaRLmda}
Let $F$ be an extension of $\mathbb{Q}_p$ with residue degree $f$ and ramification index $e$.  Let $R$ be the ring of integers of $F$ and let $\pi \in R$ be chosen so that $\pi R$ is the maximal ideal of $R$.

Let $R\Lambda = R[x]/x(x-n)R[x]$.

\begin{enumerate}
\item If $n=0$, then $\zeta_{R\Lambda}(s) = (1-p^{f(1-2s)})^{-1} (1-p^{f(-s)})^{-1}$.
\item If $n \ne 0$, then
$$\zeta_{R\Lambda}(s) = \sum_{r_2=0}^{e[n]-1} p^{fr_2(1-2s)} (1-p^{f(-s)})^{-1} + p^{fe[n](1-2s)} (1 - p^{f(-s)})^{-2}.$$
\end{enumerate}
\end{thm}

\begin{proof}
Note that $\{1,x\}$ is an $R$-basis of $R\Lambda$.  Let $I$ be an $R$-free ideal of $R\Lambda$ with finite index; i.e. $I$ is a free $R$-submodule that is an ideal of $R\Lambda$.  Since $I$ is a free $R$-submodule, it has a generating set of the form $\{\pi^{r_1}+ax,\pi^{r_2}x\}$, where $a \in R$ is given up to addition modulo $\pi^{r_2}R$.  For such a submodule, we have $[R\Lambda : I] = [R: \pi R]^{(r_1+r_2)} = p^{f(r_1+r_2)}$.

Since $I$ is an ideal of $R\Lambda$, it must be closed under multiplication by $x$, and this imposes extra conditions on our generating set.  In particular, using the fact that $x^2=nx$ in $R\Lambda$, we must have that
$$\begin{array}{rcccl}

x(\pi^{r_1}+ax) &=& (\pi^{r_1} + an)x &=& c_1(\pi^{r_1}+ax)+c_2(\pi^{r_2}x), \exists c_1,c_2 \in R, \mbox{ and } \\

x(\pi^{r_2}x) &=& (\pi^{r_2}n)x &=& d_1(\pi^{r_1}+ax)+d_2(\pi^{r_2}x), \exists d_1,d_2 \in R.

\end{array}$$
The second equation has a unique solution: $d_1=0$ and $d_2=n$. In the first equation, $c_1=0$, so it reduces to $\pi^{r_1} +an = c_2\pi^{r_2}$.  The number of distinct ideals for a fixed pair of nonnegative integers $r_1$ and $r_2$ is
$$ N(r_1,r_2) = |\{a \in R/ \pi^{r_2}R : \pi^{r_1} + an \in \pi^{r_2} R\}|. $$
When $n=0$ (which is allowed), the condition reduces to $r_1 \ge r_2$ and we have $N(r_1,r_2)=p^{fr_2}$.  On the other hand, when $n \ne 0$, we have that $nR = p^{[n]}R = \pi^{e[n]} R$.  When $0 \le r_2 < e[n]$, then we will have $r_2 \le r_1$, so every $a \in R$ satisfies $\pi^{r_1}+an \in \pi^{r_2} R$.  This means that $N(r_1,r_2) = | R : \pi^{r_2} R| = p^{fr_2}$ in this case.   However, when $e[n]\le r_2$  we will have $\pi^{r_2}/n \in \pi^{r_2}R$.  The condition then reduces to $a \in \pi^{r_1}/n + (\pi^{r_2}/n) R$.  The number of equivalence classes in $R/\pi^{r_2}R$ satisfying this condition is $p^{fr_2}/p^{fr_2 -e[n]} = p^{e[n]}$.

We can now compute the zeta function of $R\Lambda$ directly in each case.  If $n=0$, then
$$\begin{array}{rcl}
\zeta_{R\Lambda}(s) &=& \sum_{r_2 \ge 0} \sum_{r_1 \ge r_2} p^{fr_2} p^{f(r_1+r_2)(-s)} \\
&=& \sum_{r_2\ge 0} p^{fr_2(1-s)} \sum_{r_1\ge r_2} p^{fr_1(-s)} \\
&=& \sum_{r_2 \ge 0} p^{fr_2(1-s)} p^{fr_2(-s)} (1-p^{f(-s)})^{-1} \\
&=& (1-p^{f(1-2s)})^{-1} (1-p^{f(-s)})^{-1}.
\end{array}$$
When $n \ne 0$, we get
$$\begin{array}{rcl}
\zeta_{R\Lambda}(s) &=& \sum_{r_2} \sum_{r_1} N(r_1,r_2) p^{f(r_1+r_2)(-s)} \\
                    &=& \sum_{0 \le r_2 < e[n]} \sum_{r_1 \ge r_2} p^{fr_2} p^{f(r_1+r_2)(-s)} \\
		& & + \sum_{r_2 \ge e[n]} \sum_{r_1 \ge e[n]} p^{e[n]} p^{f(r_1+r_2)(-s)} \\
 		&=& \sum_{r_2=0}^{e[n]-1} p^{fr_2(1-s)} \sum_{r_1 \ge r_2} p^{fr_1(-s)} \\
		& & + p^{e[n]} \sum_{r_2 \ge e[n]} p^{fr_2(-s)} \sum_{r_1 \ge e[n]} p^{fr_1(-s)} \\
		&=& \sum_{r_2=0}^{e[n]-1} p^{fr_2(1-2s)} (1-p^{f(-s)})^{-1} + p^{fe[n](1-2s)} (1 - p^{f(-s)})^{-2}.
\end{array}$$
\end{proof}

We remark that the above approach is valid when $n=0$ because it does not require that
$\Lambda$ be contained in a maximal order of $\mathbb{Q}_p\Lambda$.   If $R$ is the ring of integers of an algebraic number field $K$, then in order to apply this approach we need to establish an Euler product formula in the absence of the semisimplicity condition.  Let $spec(R)$ be the set of nonzero prime ideals of $R$.  Suppose $\Lambda$ is an $R$-order in a nonsemisimple $K$-algebra $A$, and $L$ is a full $\Lambda$-lattice in an $A$-module $V$.  By \cite[Theorem I.8.8]{Rog70} if $M$ is a full $\Lambda$-lattice with finite index in $L$, then
$$ L/M \simeq \bigoplus_{\mathcal{P} \in spec(R)} L_{\mathcal{P}}/M_{\mathcal{P}}, $$
and $L_{\mathcal{P}}/M_{\mathcal{P}}$ can be nontrivial for only finitely many primes $\mathcal{P}$ in $spec(R)$.
It follows that $[L:M] = \prod_{\mathcal{P}} [L_{\mathcal{P}}:M_{\mathcal{P}}]$, and this product has only finitely many factors larger than $1$.  Furthermore, for any selection of full $\Lambda_{\mathcal{P}}$-lattices $\{M_{\mathcal{P}} \}$ with all but finitely many equal to $L_{\mathcal{P}}$, it follows from \cite[Lemma I.8.8]{Rog70} that there exists a full $\Lambda$-lattice $M$ with the selected $\mathcal{P}$-localization at every prime $\mathcal{P}$.  Therefore, keeping in mind these products are finite and the indices are all powers of individual rational primes $p$, we get
$$\begin{array}{rcl}
\zeta_{\Lambda}(L,s) &=& \sum_{[L:M]<\infty} [L:M]^{-s} \\
                     &=& \sum_M \big( \prod_{\mathcal{P}} [L_{\mathcal{P}}:M_{\mathcal{P}}] \big)^{-s} \\
                     &=& \prod_{\mathcal{P}} \big( \sum_{M_{\mathcal{P}}} [L_{\mathcal{P}}:M_{\mathcal{P}}]^{-s} \big) \\
		     &=& \prod_{\mathcal{P}} \zeta_{\Lambda_{\mathcal{P}}}(L_{\mathcal{P}},s).
\end{array}$$
Therefore, the factors in this Euler product can be computed using the above Lemma even in the case where $\mathbb{Q}_p \Lambda$ is not  semisimple.

\medskip
As a corollary we can generalize Hanaki and Hirasaka's calculation of the zeta function of an integral adjacency algebra of an association scheme of rank $2$ to larger coefficient rings.

\begin{corollary} \label{zetaRT}
Let $(Y,T)$ be an association scheme of rank $2$ and order $n$.  Let $p$ be a prime divisor of $n$.  Suppose $F$ is an extension of $\mathbb{Q}_p$ with residue degree $f$ and ramification index $e$, and let $R$ be its ring of integers.  Then
$$ \zeta_{RT}(s) = \sum_{r_2=0}^{e[n]-1} p^{fr_2(1-2s)} (1-p^{f(-s)})^{-1} + p^{fe[n](1-2s)} (1 - p^{f(-s)})^{-2}. $$
\end{corollary}

\begin{proof}
Under these assumptions $T = \{\sigma_0,\sigma_1\}$, where $\sigma_1$ is the adjacency matrix of the ordinary complete graph on $n$ vertices $K_n$, whose minimal polynomial is $\mu(x) = (x - (n -1))(x+1)$.  It is easy to see that $RT = R[\sigma_1]$.  By the change of variable $x \mapsto x-1$ we have $RT \simeq R[x]/x(x-n)R[x]$.  So Theorem \ref{zetaRLmda} can be applied to calculate the zeta function of $RT$ directly.
\end{proof}

\bigskip
\noindent {\bf 3. Zeta functions for tensor products of locally coprime orders }

\medskip
In this section we consider the calculation of the zeta function of the tensor product of two $\mathbb{Z}$-orders $\Lambda_1$ and $\Lambda_2$ under certain conditions.  Since $\Lambda_1$ and $\Lambda_2$ contain $\mathbb{Z}$-bases $\mathbf{b}_1$ and $\mathbf{b}_2$, we have that 
$\mathbb{Z}\Lambda_1 = \mathbb{Z}\mathbf{b}_1$ and $\mathbb{Z}\Lambda_2 = \mathbb{Z}\mathbf{b}_2$, where this notation denotes the integer span of the given set.  This means that for any extension $\mathcal{R}$ of $\mathbb{Z}$, we will have
$$\begin{array}{rcl}
\mathcal{R} \otimes_{\mathbb{Z}} (\Lambda_1 \otimes_{\mathbb{Z}} \Lambda_2) &=& \mathcal{R} \otimes_{\mathbb{Z}} \mathbb{Z}[\mathbf{b}_1 \times \mathbf{b_2}] \\
&=& \mathcal{R}[\mathbf{b}_1 \times \mathbf{b}_2] \\
&=& \mathcal{R}\mathbf{b}_1 \otimes_{\mathcal{R}} \mathcal{R}\mathbf{b}_2 \\
&=& \mathcal{R}\Lambda_1 \otimes_{\mathcal{R}} \mathcal{R}\Lambda_2.
\end{array}$$

Now for the conditions on the pair of orders $\Lambda_1$ and $\Lambda_2$, first we require that $\mathbb{Q}\Lambda_1$ and $\mathbb{Q}\Lambda_2$ are finite-dimensional commutative semisimple algebras.  This implies that $\Lambda_1$ and $\Lambda_2$ are contained in maximal orders $\Gamma_1$ and $\Gamma_2$, respectively, and we can write
$$ \Gamma_i \simeq \oplus_{j=1}^{h_i} R_{ij}, i = 1,2,$$
where $R_{ij}$ is the ring of integers in the algebraic number field $F_{ij}$ that appears as the center of the $j$-th simple component of $\mathbb{Q}\Lambda_i$.  For a given rational prime $p$, we will assume $F_{ij}$ has splitting degree $g_{ij}$, that $\mathcal{P}_{ij1}, \dots, \mathcal{P}_{ijg_{ij}}$ are the distinct primes of $F_{ij}$ lying above $p$, and that $R_{ijk}$ is the completion $(R_{ij})_{\mathcal{P}_{ijk}}$ for $k=1,\dots,g_{ij}$.

Second, we will require that the two orders $\Lambda_1$ and $\Lambda_2$ 
are {\it locally coprime}, which means that for all primes $p$, either $\mathbb{Z}_p\Lambda_1$ is a maximal order in $\mathbb{Q}_p\Lambda_1$ or $\mathbb{Z}_p\Lambda_2$ is a maximal order in $\mathbb{Q}_p\Lambda_2$.  Under these two conditions we will obtain an expression for $\zeta_{\mathbb{Z}_p}[\Lambda_1 \otimes \Lambda_2]$ for all primes $p$.

The zeta function for $\Lambda_1 \otimes_{\mathbb{Z}} \Lambda_2$ will be expressed using the Euler product formula: if $\Lambda$ is an order in a semisimple algebra and $\Gamma_p$ is a maximal order of $\mathbb{Q}_p\Lambda$ containing $\mathbb{Z}_p\Lambda$ for all rational primes $p$, then
$$ \zeta_{\Lambda}(s) = \zeta_{\Gamma}(s) \prod_{p \in B} \frac{\zeta_{\mathbb{Z}_p\Lambda}(s)}{\zeta_{\Gamma_p}(s)}, $$
where $B$ is the set of primes $p$ for which $\mathbb{Z}_p\Lambda$ is not a maximal order of $\mathbb{Q}_p\Lambda$.

\begin{thm} \label{PCoprime}
Let $\Lambda_1$ and $\Lambda_2$ be orders in commutative semisimple $\mathbb{Q}$-algebras $\mathbb{Q}\Lambda_1$ and $\mathbb{Q}\Lambda_2$, respectively.  Suppose that $\Lambda_1$ and $\Lambda_2$ are locally coprime.   Let $\Gamma_1 = \oplus_{j=1}^{h_1} R_{1j}$ and $\Gamma_2 = \oplus_{j=1}^{h_2} R_{2j}$ be maximal orders containing $\Lambda_1$ and $\Lambda_2$.  For each rational prime $p$, we have $g_{ij}, R_{ijk}, f_{ijk},$ and $e_{ijk}$ as defined above, so in particular $R_{ijk}$ is a direct summand of the maximal order $\mathbb{Q}_p\Gamma_i$ of $\mathbb{Q}_p\Lambda_i$ for all primes $p$ and $i=1,2$.  Let $B_i$ be the set of primes $p$ for which $\mathbb{Q}_p\Lambda_i \neq \mathbb{Q}_p\Gamma_i$ for $i=1,2$.  Then
$$ \zeta_{\mathbb{Z}[\Lambda_1 \otimes \Lambda_2]}(s) = \zeta_{\Gamma_1 \otimes \Gamma_2}(s) \delta_{B_1}(s) \delta_{B_2}(s), $$
where
$$ \delta_{B_1}(s) = \prod_{p \in B_1} \prod_{j=1}^{h_2} \prod_{k=1}^{g_{2j}} \frac{\zeta_{R_{2jk}\Lambda_1}(s)}{\zeta_{R_{2jk}\Gamma_1}(s)}, $$
and
$$ \delta_{B_2}(s) = \prod_{p \in B_2} \prod_{j=1}^{h_1} \prod_{k=1}^{g_{1j}} \frac{\zeta_{R_{1jk}\Lambda_2}(s)}{\zeta_{R_{1jk}\Gamma_2}(s)}. $$
\end{thm}

\begin{proof}
By the Euler product formula, it suffices to verify the formula for $\delta_{B_1}(s)$.  If $p \in B_1$, then $\mathbb{Z}_p\Lambda_2$ is a maximal order in $\mathbb{Q}_p\Lambda_2$, so our notation says that
$$\mathbb{Z}_p\Lambda_2 \simeq \oplus_{j=1}^{h_2} \oplus_{k=1}^{g_{2j}} R_{2jk}.$$
Now, we have that
$$\begin{array}{rcl}
\mathbb{Z}_p \otimes [\Lambda_1 \otimes \Lambda_2] &\simeq& \mathbb{Z}_p \Lambda_1 \otimes_{\mathbb{Z}_p} \mathbb{Z}_p\Lambda_2 \\
&\simeq& \mathbb{Z}_p\Lambda_1 \otimes_{\mathbb{Z}_p} ( \oplus_{j=1}^{h_2} \oplus_{k=1}^{g_{2j}} R_{2jk} ) \\
&=& \oplus_{j=1}^{h_2} \oplus_{k=1}^{g_{2j}} R_{2jk} \Lambda_1.
\end{array}$$
It follows that
$$\zeta_{\mathbb{Z}_p[\Lambda_1 \otimes \Lambda_2]}(s) = \prod_{j=1}^{h_2} \prod_{k=1}^{g_{2j}} \zeta_{R_{2jk}\Lambda_1}(s).$$
The formula for $\delta_{B_1}(s)$ is now a consequence of the Euler product formula.
\end{proof}

It is straightforward to extend this idea to the tensor product of finitely many orders that are locally coprime, where this is taken to mean that for any rational prime $p$, at most one of the $p$-adic completions of the orders is not a maximal order in its overlying $\mathbb{Q}_p$-algebra.

\medskip
The {\it direct product} of $(X,S)$ and $(Y,T)$ is the scheme $(X \times Y, S \times T)$ whose adjacency matrices are the pairwise tensor products of the adjacency matrices of $S$ with those of $T$. (see \cite[\S 7]{Ziesh}). This produces a $\mathbb{Q}$-algebra with a predictable Wedderburn decomposition since
$$ \mathbb{Q}[S \times T] \simeq \mathbb{Q}S \otimes_{\mathbb{Q}} \mathbb{Q}T. $$
Since $\mathbb{Z}S$ and $\mathbb{Z}T$ are orders with
$$ \mathbb{Z}[S \times T] \simeq \mathbb{Z}S \otimes_{\mathbb{Z}} \mathbb{Z}T, $$
it is clear that the approach of section 3 can be directly applied in cases where we have two commutative association schemes whose integral scheme rings are locally coprime.  We will give several examples where this does occur in the next section.

\bigskip
\noindent {\bf 4. Examples of explicit zeta functions }

\medskip
In this section we apply the results and methods of section 2 and 3 to produce explicit zeta functions in several new instances.  Our first application combines the notation of section 3 with Corollary \ref{zetaRT}.

\begin{proposition}
Let $(Y,T)$ be an association scheme of rank $2$ and order $n$.
Let $R$ be the ring of integers in an algebraic number field $F$.  For each prime $p$, let $g_p$ be the splitting degree of $F$ at $p$, let $R_{p1}, \dots, R_{pg_p}$ be the distinct $p$-adic completions of $R$ at the primes of $F$ lying over $p$, and suppose the quotient field of each $R_{pj}$ has residue degree $f_{pj}$ and ramification index $e_{pj}$.   Then
$$ \zeta_{RT}(s) = \zeta_{R}(s)^2 \prod_{p|n} \prod_{j=1}^{g_p} \dfrac{\zeta_{R_{pj}T}(s)}{\zeta_{R_{pj}}(s)^2}, $$
where $\zeta_{R_{pj}T}(s)$ is calculated for the $p$-adic ring of integers $R_{pj}$ using the formula of Corollary \ref{zetaRT} in terms of $n$, $e_{pj}$, and $f_{pj}$.
\end{proposition}

\begin{proof}
This is a consequence of the Euler product formula.  The set of rational primes $B$ where $\mathbb{Z}T$ is not a maximal order of $\mathbb{Q}T$ is the set of divisors of $n$.  This means the set of primes $\mathcal{P}$ of $F$ for which $RT$ is not a maximal order of $FT$ is contained in the set of primes of $F$ lying over a prime dividing $n$.
The square of the Dedekind zeta function occurs because the maximal order of $FT$ is the direct sum of two copies of $R$.   The square in the denominator occurs because the maximal order containing $R_{pj}T$ is a direct sum of two copies of $R_{pj}$.
\end{proof}

For our main applications we give several examples of zeta functions for locally coprime pairs of integral adjacency algebras of association schemes among those whose zeta functions were known previously.

\begin{example} {\rm
Consider the special case of the zeta function of the integral adjacency algebra of the direct product $S \times T$ of two association schemes in which the second scheme is the group $C_2$.  In this case in the notation of section 3 we have $B_2 = \{2\}$ and $\Gamma_2 = \mathbb{Z} \oplus \mathbb{Z}$. Let $(X,S)$ be any commutative scheme for which $\mathbb{Z}_2S$ is a maximal order of $\mathbb{Q}_2S$.   Using the notation of section 3, write the Wedderburn decomposition of $\mathbb{Q}_2S$ as $\oplus_{j=1}^{h_1} \oplus_{k=1}^{g_{1j}} F_{1jk}$. Then in the notation of Theorem \ref{PCoprime} we have
$$ \delta_{B_2}(s) = \prod_{j=1}^{h_1} \prod_{k=1}^{g_{1j}} (1-2^{-f_{1jk}s}+2^{-f_{1jk}(1-2s)}). $$ 

For an easy explicit example, we will give the zeta function of $\mathbb{Z}C_6 \simeq \mathbb{Z}[C_3 \times C_2]$ (though it should be noted that Hironaka's results will give this by essentially the same method).  In this case $B_1 = \{ 3 \}$ and $B_2=\{2\}$ so $\mathbb{Z}C_3$ and $\mathbb{Z}C_2$ are locally coprime.  The Wedderburn decomposition of $\mathbb{Q}[C_3 \times C_2]$ is $\mathbb{Q} \oplus \mathbb{Q} \oplus \mathbb{Q}(\varepsilon_3) \oplus \mathbb{Q}(\varepsilon_3)$.  $\mathbb{Q}_2(\varepsilon_3)$ is unramified of degree $f=2$.  Therefore, from Theorem \ref{PCoprime} and Corollary \ref{zetaRT}, we have

$$\begin{array}{rcl}
\zeta_{\mathbb{Z}[C_3 \times C_2]}(s) = \zeta_{\mathbb{Z}}(s)^2 \zeta_{\mathbb{Z}[\varepsilon_3]}(s)^2  &\times& (1-3^{-s}+3^{(1-2s)})(1-3^{-2s}+3^{2(1-2s)}) \\
&\times& (1-2^{-s}+2^{(1-2s)})^2.
\end{array}$$

We can even say something about the $2$-adic zeta function of $S \times C_2$ when the association scheme $S$ is not commutative, as long as $\mathbb{Z}_2S$ is a maximal order of $\mathbb{Q}_2S$.  If the Wedderburn decomposition of $\mathbb{Q}_2S$ is
$$ \mathbb{Q}_2S \simeq \bigoplus_{j=1}^{h_2} \bigoplus_{k=1}^{g_{1j}} M_{r_{1jk}}(D_{1jk}),$$ and the $2$-adic divison algebras $D_{1jk}$ have index $m_{1jk}$ and center $F_{1jk}$, then the fact that the maximal order containing $\mathbb{Z}_2T$ is the direct sum of two copies of $\mathbb{Z}_2$ implies that
$$ \zeta_{\mathbb{Z}_2[S \times T]}(s) = \prod_{j=1}^{h_1} \prod_{k=1}^{g_{1j}} (1-2^{-f_{1jk}s}+2^{-f_{1jk}(1-2s)}) \zeta_{\Gamma_{1jk}}(s)^2, $$
where $\Gamma_{1jk}$ is a maximal order of the local central simple algebra $M_{r_{1jk}}(D_{1jk})$, whose zeta function can be calculated using Hey's formula.
}\end{example}

\begin{example} {\rm Next we give the zeta function of $\mathbb{Z}[S \times T]$, where $S$ and $T$ are association schemes of rank $2$ having coprime orders $m$ and $n$, respectively.   These scheme rings are the integral adjacency algebras of the ordinary complete graphs $K_m$ and $K_n$ and it is convenient to use this notation.   In this case $B_1$ is the set of primes dividing $m$, and $B_2$ is the set of primes dividing $n$, so the assumption $(m,n)=1$ implies that $\mathbb{Z}K_m$ and $\mathbb{Z}K_n$ are locally coprime.   The Wedderburn decomposition of $\mathbb{Q}[K_m \times K_n]$ consists of the direct sum of $4$ copies of $\mathbb{Q}$, so by a direct application of Theorem \ref{PCoprime}, we find

$$\begin{array}{l}
\zeta_{\mathbb{Z}[K_m \times K_n]}(s) = \zeta_{\mathbb{Z}}(s)^4 \\
\\
\qquad \times  \prod_{p | m} \bigg(\dfrac{(\sum_{r_2=0}^{[m]_p-1} p^{r_2(1-2s)} (1-p^{(-s)})^{-1}) + p^{[m]_p(1-2s)} (1 - p^{(-s)})^{-2}}{(1-p^{-s})^{-2}}\bigg)^2 \\
\\
\qquad \times  \prod_{q | n}  \bigg( \dfrac{(\sum_{r_2=0}^{[n]_q-1} q^{r_2(1-2s)} (1-q^{(-s)})^{-1}) + q^{[n]_q(1-2s)} (1 - q^{(-s)})^{-2}}{(1-q^{-s})^{-2}} \bigg)^2 \\
\\
\quad = \zeta_{\mathbb{Z}}(s)^4 \times  \prod_{p | m} \bigg((\sum_{r_2=0}^{[m]_p-1} p^{r_2(1-2s)} (1-p^{-s})) + p^{[m]_p(1-2s)} \bigg)^2 \\
\\
\qquad \times  \prod_{q | n}  \bigg( (\sum_{r_2=0}^{[n]_q-1} q^{r_2(1-2s)} (1-q^{-s})) + q^{[n]_q(1-2s)} \bigg)^2 \\
\\
.
\end{array}$$
where $p$ runs over prime divisors of $m$ and $q$ runs over prime divisors of $n$.

This formula generalizes immediately to the product of any number of rank $2$ schemes of pairwise coprime orders.
 }\end{example}

\begin{example} {\rm
Finally, we give the zeta function of $\mathbb{Z}[C_p \times K_n]$, where $\mathbb{Z}[K_n]$ is the adjacency algebra of the rank $2$ scheme of order $n$ and $p$ is a prime that does not divide $n$.  In this case the Wedderburn decomposition of $\mathbb{Q}[C_p \times K_n]$ is
$\mathbb{Q}[C_p \times K_n] \simeq \mathbb{Q} \oplus \mathbb{Q} \oplus \mathbb{Q}(\varepsilon_p) \oplus \mathbb{Q}(\varepsilon_p)$.   Note that, in our notation we will have $h_1=h_2=2$.  Suppose that for each prime divisor $q$ of $n$, $\mathbb{Q}_q(\varepsilon_p)$ has residue degree $f_q$.  Then there are $g_q=(p-1)/f_q$ primes of $\mathbb{Z}_q[\varepsilon_p]$ lying over $q$.  By Theorem \ref{PCoprime}, we have
$$\begin{array}{l}
\zeta_{\mathbb{Z}[C_p \times K_n] }(s) =  \zeta_{\mathbb{Z}}(s)^2  \zeta_{\mathbb{Z}[\varepsilon_p]}(s)^2  (1 - p^{-s} + p^{1-2s})^2 \prod_{q|n} \prod_{j=1}^{h_1} \prod_{k=1}^{g_1j} \dfrac{\zeta_{R_{1jk}K_n}(s)}{\zeta_{R_{1jk}}(s)^2} \\
 \\
= \zeta_{\mathbb{Z}}(s)^2  \zeta_{\mathbb{Z}[\varepsilon_p]}(s)^2  (1 - p^{-s} + p^{1-2s})^2 \\
\\
\times  \prod_{q|n}  \bigg(\dfrac{(\sum_{r_2=0}^{[n]_q-1} q^{r_2(1-2s)}) (1-q^{-s})^{-1} + q^{[n]_q(1-2s)} (1 - q^{-s})^{-2}}{(1-q^{-s})^{-2}} \\
\\
\times  \prod_{k=1}^{(p-1)/f_q}  \dfrac{(\sum_{r_2=0}^{[n]_q-1} q^{f_qr_2(1-2s)}) (1-q^{-f_qs})^{-1} + q^{f_q[n]_q(1-2s)} (1 - q^{-f_q-s})^{-2})}{(1-q^{-f_qs})^{-2}} \bigg) \\
\\
 = \zeta_{\mathbb{Z}}(s)^2  \zeta_{\mathbb{Z}[\varepsilon_p]}(s)^2  (1 - p^{-s} + p^{1-2s})^2 \\
\\
\quad \times  \prod_{q|n}  \bigg(\big((\sum_{r_2=0}^{[n]_q-1} q^{r_2(1-2s)}) (1-q^{-s}) + q^{[n]_q(1-2s)} \big)\\
\\
\qquad \times \big( \prod_{k=1}^{(p-1)/f_q} (\sum_{r_2=0}^{[n]_q-1} q^{f_qr_2(1-2s)}) (1-q^{-f_qs}) + q^{f_q[n]_q(1-2s)} \big) \bigg).
\end{array}$$
Again it should be straightforward to generalize this formula to give the zeta function of $\mathbb{Z}[C_m \times T]$, where $T$ is a rank $2$ scheme of order $n$ and $m$ is a square-free positive integer that is relatively prime to $n$.
}\end{example}

Zeta functions have been described for $C_p \times C_p$ for a prime $p$ \cite{Tak92}, $C_{p^2}$ for a prime $p$ \cite{Reiner} and nonabelian metacyclic groups of type $C_p \rtimes C_q$ for a prime $p$ and a prime $q$ dividing $p-1$ \cite{Hiron81}.  The formulas for these are a bit more complicated, but in all of these cases the sets of ``bad'' primes are precisely the primes dividing the order of the group.  The direct product of any of these groups with a cyclic group of coprime square-free coprime order or a rank $2$ scheme of coprime order can be handled using the present methods.  For example, the zeta functions of the integral scheme rings of $(C_2 \times C_2) \times K_3$, $C_4 \times C_3$, or $S_3 \times K_5$ can be established from their results using an application of Theorem \ref{PCoprime}.

\medskip
The explicit formulas for the zeta function of $C_2 \times C_2$ and $C_3 \times C_3$ presented by Takagehara \cite{Tak92} indicate that it will be considerably more difficult to calculate the zeta function of the tensor product of two integral adjacency algebras that are not locally coprime.


\begin{thebibliography}{10}

\bibitem{BR80} C. Bushnell and I. Reiner, Zeta functions of arithmetic orders and Solomon's conjectures, {\it Math. Z.}, {\bf 173} (1980), 135-161.

\bibitem{hanakisemi}  A. Hanaki, Semisimplicity of adjacency algebras of association schemes,
{\it J. Algebra}, {\bf 225} (2000), no. 1, 124-129.

\bibitem{HH14} A. Hanaki and M. Hirasaka, The zeta function of the integral adjacency algebra of association scheme of rank $2$, {\it Hokkaido Math. J.}, to appear.

\bibitem{Hiron81} Y. Hironaka, Zeta functions of integral group rings of metacyclic groups, {\it Tskuba Math. J}, {\bf 5} (2), (1981), 267-283.

\bibitem{Reiner} I. Reiner, Zeta functions of integral representations, {\it Comm. Algebra}, {\bf 8} (10), (1980), 911-925.

\bibitem{Rog70} K. W. Roggenkamp and V. Huber-Dyson, Lattices over Orders, I, {\it Lecture Notes in Mathematics}, No. 115, Springer-Verlag, 1970.

\bibitem{Sol77} L. Solomon, Zeta functions and integral representation theory, {\it Adv. Math.} {\bf 26} (1977), 306-326.

\bibitem{Tak92} Y. Takegahara, Zeta functions of integral group rings of abelian ($p$,$p$)-groups, {\it Comm. Algebra}, {\bf 15} (12), (1992), 2565-2615.

\bibitem{Ziesh} P.-H. Zieschang, {\it Theory of Association Schemes}, Springer-Verlag, 2005.
\end{thebibliography}
\end{document}